\documentclass[11pt]{article}

\usepackage{graphicx}

\usepackage{amsmath, amsthm, amssymb, tikz}
\usepackage{amsfonts}%
\usepackage{fullpage}
\usepackage[enableskew,vcentermath]{youngtab}

\newtheorem{theorem}{Theorem}[section]
\newtheorem{proposition}[theorem]{Proposition}

\newtheorem{corollary}[theorem]{Corollary}

\newtheorem{definition}[theorem]{Definition}
\newtheorem{claim}[theorem]{Claim}
\newtheorem{remark}[theorem]{Remark}

\newtheorem{problem}[theorem]{Problem}

\newcounter{Examplecount}
\newenvironment{example}[1][Example \arabic{Examplecount}.]{\begin{trivlist}
\item[\hskip \labelsep {\bfseries
#1}]}{\end{trivlist}\stepcounter{Examplecount}}

\newcommand\beq{\begin{equation}}
\newcommand\eeq{\end{equation}}
\newcommand\bce{\begin{center}}
\newcommand\ece{\end{center}}
\newcommand\bea{\begin{eqnarray}}
\newcommand\eea{\end{eqnarray}}
\newcommand\ba{\begin{array}}
\newcommand\ea{\end{array}}
\newcommand\ben{\begin{enumerate}}
\newcommand\een{\end{enumerate}}
\newcommand\bit{\begin{itemize}}
\newcommand\eit{\end{itemize}}
\newcommand\brr{\begin{array}}
\newcommand\err{\end{array}}
\newcommand\bt{\begin{tabular}}
\newcommand\et{\end{tabular}}

\newcommand\nn{\nonumber}

\newcommand\wt{\widetilde}

\renewcommand\S{{\mathcal S}}

\DeclareMathOperator\maj{maj} 

\newcommand\x{{\mathbf x}}
\newcommand\y{{\mathbf y}}

\newcommand{\bbz}{\mathbb{Z}}

\renewcommand\S{{\mathcal S}}
\newcommand\s{{\tt S}}
\newcommand\A{{\mathcal A}}
\newcommand\LL{{\mathcal L}}

\newcommand{\D}{{\rm{Des}}}
\newcommand{\inv}{{\rm{inv}}}
\newcommand{\des}{{\rm{des}}}
\newcommand{\Neg}{{\rm{Neg}}}
\newcommand{\nega}{{\rm{neg}}}
\newcommand{\fmaj}{{\rm{fmaj}}}
\newcommand{\fdes}{{\rm{fdes}}}
\newcommand{\sgn}{{\rm{sign}}}
\renewcommand{\O}{\mathcal{O}}
\newcommand{\As}{\A^s}   
\newcommand{\AB}{\A^B}

\begin{document}

\title{Signed arc permutations}


\author{Sergi Elizalde~\thanks{Department of Mathematics, Dartmouth College, Hanover, NH 03755, USA. {\tt sergi.elizalde@dartmouth.edu}.
Partially supported by NSF grant DMS-1001046.} \and Yuval
Roichman~\thanks{Department of Mathematics, Bar-Ilan University,
 Ramat-Gan 52900, Israel.  {\tt yuvalr@math.biu.ac.il}. }}


\maketitle

\begin{abstract}
Arc permutations, which were originally introduced in the study of triangulations and characters, have recently been shown to have interesting combinatorial properties.
The first part of this paper continues their study by providing signed enumeration formulas with respect to their descent set and major index.
Next, we generalize the notion of arc permutations to the hyperoctahedral group in two different directions.
We show that these extensions to type $B$ carry interesting analogues of the properties of type $A$ arc permutations,
such as characterizations by pattern avoidance, and elegant unsigned and signed enumeration formulas with respect to the flag-major index. 


\end{abstract}

\tableofcontents

\section{Introduction}

The enumeration of permutations taking into account their sign, usually referred to as {\em signed enumeration},
was studied for subsets of the symmetric group $\S_n$ in the seminal paper of Simion and Schmidt on pattern-avoiding
permutations~\cite{SiS}.
Among many other instances of sign enumeration in the literature,
we highlight an elegant formula for the signed descent number
enumerator
conjectured by Loday~\cite{Lod} and proved by D\'esarmenien and
Foata~\cite{DF} and by Wachs~\cite{Wachs}. Type $B$ analogues were
given later by Reiner~\cite{Rein}.

In analogy to MacMahon's well-known product formula enumerating permutations in $\S_n$ with respect to the major index,
a factorial-type formula for the signed major index enumerator on $\S_n$ was given by Gessel and Simion~\cite[Cor.~2]{Wachs}.
For generalizations to other groups, see~\cite{AGR, BRS, CGH, BZ,
BC, C}. In this paper we study signed major index enumerators and other related polynomials for arc permutations, both in the symmetric group
$\S_n$ and in the hyperoctahedral group $B_n$.

{\em Arc permutations} were introduced in \cite{TFT2} as a subset of the symmetric group. These permutations play an important role in the study of flip
graphs of polygon triangulations and associated affine Weyl group actions. It was shown in \cite{ER} that arc permutations can be characterized in terms of
pattern avoidance. A descent-set preserving map from
arc permutations to Young tableaux was constructed in~\cite{ER} to deduce a
conjectured character formula of Regev.

\medskip

In this paper we propose two different generalizations of the
notion of arc permutations to the hyperoctahedral group. These
generalizations, which we call {\em signed arc permutations} and
{\em $B$-arc permutations}, carry known properties of unsigned arc
permutations and reveal new ones. In particular, we give
characterizations of both generalizations by forbidden patterns
(see Theorems~\ref{signed_arc_pattern} and~\ref{B_arc_pattern}),
in analogy to the results from~\cite{ER} in the unsigned case.
Additionally, we show in Subsection~\ref{sec:another_approach}
that both unsigned arc permutations and $B$-arc permutations may
be characterized by their canonical expressions. This
characterization will be used to derive the signed and unsigned
flag-major index enumerators for $B$-arc permutations. In the case
of signed arc permutations, different tools are used to derive
similar formulas in Section~\ref{sec:signed-arc}.

For both generalizations of arc permutations to type $B$, we obtain nice product formulas for their
descent set enumerators (see Theorems~\ref{signed_Des_Neg_inv} and~\ref{Thm_AB_Des}).
Even though the descent set has a different distribution on these two definitions, 
it turns out that they both carry the same unsigned and signed flag-major index enumerators (see
Corollary~\ref{equi-B_vs_signed}). This surprising phenomenon deserves further study.

\section{Arc permutations in the symmetric group}

\subsection{Definition and basic properties}

We start by reviewing two definitions and a result from~\cite{ER}.
Recall that an interval of $\bbz_n$ is a set of the form
$\{a,a+1,\dots,b\}$ or $\{b,b+1,\dots,n,1,2,\dots,a\}$ where
$1\le a\le b\le n$.

\begin{definition}\label{def_arc} A permutation $\pi\in \S_n$ is an {\em arc permutation}
if, for every $1\le j\le n$, the first $j$ letters in $\pi$ form
an interval in $\bbz_n$. 
Denote by $\A_n$ the set of arc permutations in $\S_n$.

A permutation $\pi\in \A_n$ is {\em left-unimodal} 
if, for every $1\le j\le n$, the first 
$j$ letters in $\pi$ form an interval in
$\bbz$. Denote  by $\LL_n$ the set of left-unimodal permutations
in $\S_n$.
\end{definition}

\begin{example} We have that $12543\in\A_5$, but $125436\notin\A_6$, since
$\{1,2,5\}$ is an interval in $\bbz_5$ but not in $\bbz_6$.
\end{example}

It is easy to show~\cite{ER} that $|\A_n|=n2^{n-2}$ for $n\ge2$.
Arc permutations can be characterized in terms of pattern
avoidance, as those permutations avoiding the eight patterns
$\tau\in\S_4$ with $|\tau(1)-\tau(2)|=2$.

\begin{theorem}[\cite{ER}]\label{arc-pattern-thm}
$$\A_n=\S_n(1324,1342,2413,2431,3124,3142,4213,4231).$$
\end{theorem}

\subsection{Enumeration}

For a permutation $\pi\in \S_n$, recall the definition of its descent set
$$
\D(\pi):=\{i:\ \pi(i)>\pi(i+1)\},
$$
its major index
$$
\maj(\pi):=\sum\limits_{i\in \D(\pi)} i,
$$
and its inversion number
$$
\inv(\pi):=\#\{i<j:\ \pi(i)>\pi(j)\}.
$$
For a set $D=\{i_1,\dots,i_k\}$ denote $\x^{D}=x_{i_1}\cdots
x_{i_k}$.

\begin{theorem}\label{thm:A-invDes}
For every $n\ge2$,
$$
\sum_{\pi\in\A_n} t^{\inv(\pi)} \x^{\D(\pi)}=
\prod_{i=1}^{n-1}(1+t^i x_i)+
\sum_{j=1}^{n-2}\left((t^{j(n-j)}x_j+t^{n-j-1}x_{j+1})\prod_{i=1}^{j-1}(1+t^i
x_i)\prod_{i=j+2}^{n-1}(1+t^{n-i} x_i)\right).
$$
\end{theorem}

\begin{proof}
We separate permutations $\pi\in\A_n$ into those that are left-unimodal and those that are not.

Left-unimodal permutations are in bijection with subsets of $[n-1]$, the bijection given by taking their descent set.
Thus, such permutations are determined by choosing, for each $1\le i\le n-1$, whether $\pi(i)>\pi(i+1)$ or
$\pi(i)<\pi(i+1)$. In the first case, we introduce a descent in position $i$, and inversions between $\pi_{i+1}$ and all the preceding entries of $\pi$,
contributing $t^ix_i$ to the generating function, while in the second case no descents or inversions are created.
It follows that left-unimodal permutations contribute
$\prod_{i=1}^{n-1}(1+t^i x_i)$ to the generating function.

If $\pi$ is not left-unimodal, let $j$ be the largest such that $\{\pi(1),\dots,\pi(j)\}$ is an interval in $\bbz$. Note that $1\le j\le n-2$, and that $\pi(j+1)\in\{1,n\}$.

If $\pi(j+1)=1$, then the first $j$ entries in $\pi$ are larger than the last $n-j$ entries, creating $j(n-j)$ inversions and a descent in position $j$. For each $i$ with $1\le i\le j-1$ or $j+2\le i\le n-1$, we have the choice of whether $\pi(i)>\pi(i+1)$ or $\pi(i)<\pi(i+1)$.
For $1\le i\le j-1$, the first option introduces a descent in position $i$ and inversions between
$\pi_{i+1}$ and the entries to its left, contributing $t^ix_i$. Similarly, for $j+2\le i\le n-1$, the choice
$\pi(i)>\pi(i+1)$ introduces inversions between $\pi_{i}$ and the entries to its right, contributing $t^{n-i}x_i$. In total, the contribution of non-left-unimodal permutations with $\pi(j+1)=1$ is
$$t^{j(n-j)}x_j\prod_{i=1}^{j-1}(1+t^i x_i)\prod_{i=j+2}^{n-1}(1+t^{n-i} x_i).$$

If $\pi(j+1)=n$, the argument is similar, except that instead of a descent in position $j$ there is a descent in position $j+1$, and there are inversions between $\pi(j+1)$ and the entries to its right, so the contribution in this case is
$$t^{n-j-1}x_{j+1}\prod_{i=1}^{j-1}(1+t^i x_i)\prod_{i=j+2}^{n-1}(1+t^{n-i} x_i).$$
\end{proof}

Substituting $t=1$ in Theorem~\ref{thm:A-invDes} we recover the
following formula from~\cite{ER}:
\beq\label{eq:A-Des}
\sum_{\pi\in\A_n}\x^{\D(\pi)}=\prod\limits_{i=1}^{n-1}(1+x_{i})\left(1+\sum_{j=1}^{n-2}\frac{x_j+x_{j+1}}{(1+x_j)(1+x_{j+1})}\right)
\eeq
for every $n\ge 2$. It is now easy to obtain the $(\des,\maj)$-enumerator for arc permutations.
Recall the notation
$[n]_q=1+q+q^2+\dots+q^{n-1}=\frac{1-q^n}{1-q}$.

\begin{corollary}\label{arc_maj_des}
For every $n\ge2$,
$$
\sum\limits_{\pi\in \A_n} t^{\des(\pi)} q^{\maj(\pi)}=
\prod\limits_{i=2}^{n-2}(1+tq^i) \left(1+2tq[n-1]_{q}+t^2q^n\right).
$$
In particular,
$$
\sum\limits_{\pi\in \A_n} t^{\des(\pi)}=(1+t)^{n-3}\left(1+2(n-1)t+t^2\right).
$$
\end{corollary}

\begin{proof}
Substituting $x_i=tq^i$ for $1\le i\le n-1$ in
Equation~\eqref{eq:A-Des}, we get
$$\sum_{\pi\in\A_n} t^{\des(\pi)} q^{\maj(\pi)}=
\prod\limits_{i=1}^{n-1}(1+tq^i)\left(1+\sum_{j=1}^{n-2}\frac{tq^j(1+q)}{(1+tq^j)(1+tq^{j+1})}\right).$$
Using that
$$\frac{tq^{j}}{(1+tq^{j})(1+tq^{j+1})}=\frac{1}{1-q}\left(\frac{1}{1+tq^{j+1}}-\frac{1}{1+tq^{j}}\right),$$
the summation on the right-hand side becomes a telescopic sum that
simplifies to
$$\frac{(1+q)t q[n-2]_q}{(1+tq)(1+tq^{n-1})},$$
from where the first formula in the statement follows. The second formula is obtained by
substituting $q=1$.
\end{proof}

%




\bigskip

Now we turn to signed enumeration of arc permutations. Recall that $\sgn(\pi)=(-1)^{\inv(\pi)}$. Setting $t=-1$ in
Theorem~\ref{thm:A-invDes},
we get
\begin{multline}\label{eq:sgnDes}
\sum_{\pi\in\A_n} \sgn(\pi) \x^{\D(\pi)}= \prod_{i=1}^{n-1}(1+(-1)^i x_i)\\
+\sum_{j=1}^{n-2}\left(((-1)^{j(n-j)}x_j+(-1)^{n-j-1}x_{j+1})\prod_{i=1}^{j-1}(1+(-1)^i
x_i)\prod_{i=j+2}^{n-1}(1+(-1)^{n-i}
x_i)\right).
\end{multline}
When $n$ is even, this formula simplifies to
\begin{equation}\label{eq:sgnDeseven}
\sum_{\pi\in\A_n} \sgn(\pi) \x^{\D(\pi)}=
\prod_{i=1}^{n-1}(1+(-1)^i
x_i)\left(1+\sum_{j=1}^{n-2}\frac{(-1)^j(x_j-x_{j+1})}{(1+(-1)^j
x_j)(1+(-1)^{j+1} x_{j+1})}\right).
\end{equation}

\bigskip

\begin{theorem}\label{A:signed-maj}
For every $n\ge 2$
\begin{align}
\nn\sum_{\pi\in\A_n} q^{\maj(\pi)}&=[n]_q \prod\limits_{i=1}^{n-2} (1+q^{i}),\\
\nn\sum\limits_{\pi\in \A_n} \sgn(\pi) q^{\maj(\pi)}&=
[n]_{(-1)^{n-1}q} \prod\limits_{i=1}^{n-2} (1+ (-q)^i).
\end{align}
\end{theorem}

\begin{proof}
Substituting $t=1$ in Corollary~\ref{arc_maj_des} gives the first
formula, which already appears   in~\cite[Cor. 7]{ER}. To prove the second formula, we consider two cases
depending on the parity of $n$.

\medskip

If $n$ is even, substituting $x_i=q^i$ for $1\le i\le n-1$ in
Equation~\eqref{eq:sgnDeseven} gives
$$\sum_{\pi\in\A_n} \sgn(\pi) q^{\maj(\pi)}= \prod_{i=1}^{n-1}(1+(-q)^i)\left(1+\sum_{j=1}^{n-2}\frac{(-1)^j(q^j-q^{j+1})}{(1+(-q)^j)(1+(-q)^{j+1})}\right).$$
Letting $z=-q$, the formula becomes
\beq\label{eq:sgnDesz}\prod_{i=1}^{n-1}(1+z^i)\left(1+\sum_{j=1}^{n-2}\frac{z^j+z^{j+1}}{(1+z^j)(1+z^{j+1})}\right),\eeq
where the sum can be simplified as
$$\frac{1+z}{1-z}\sum_{j=1}^{n-2}\left(\frac{1}{1+z^{j+1}}-\frac{1}{1+z^j}\right)
=\frac{1+z}{1-z}\left(\frac{1}{1+z^{n-1}}-\frac{1}{1+z}\right)=\frac{z-z^{n-1}}{(1-z)(1+z^{n-1})},
$$
and so Equation~\eqref{eq:sgnDesz} equals
$$\prod_{i=1}^{n-1}(1+z^i)\left(1+\frac{z-z^{n-1}}{(1-z)(1+z^{n-1})}\right)=\prod_{i=1}^{n-2}(1+z^i)[n]_z=[n]_{-q}\prod_{i=1}^{n-2}(1+(-q)^i).$$

If $n$ is odd, substituting $x_j=q^j$ in
Equation~\eqref{eq:sgnDes} gives
\begin{eqnarray}\nn\sum_{\pi\in\A_n} \sgn(\pi) q^{\maj(\pi)}&=& \prod_{i=1}^{n-1}(1+(-q)^i)+\sum_{j=1}^{n-2}\left( (q^j+(-1)^jq^{j+1})\prod_{i=1}^{j-1}(1+(-q)^i) \prod_{i=j+2}^{n-1}(1-(-q)^{i})\right)\\
\label{eq:sgnmaj} &=&
\prod_{i=1}^{n-1}(1+z^i)+\sum_{j=1}^{n-2}\left(((-1)^j
z^j-z^{j+1})\prod_{i=1}^{j-1}(1+z^i)
\prod_{i=j+2}^{n-1}(1-z^{i})\right),
\end{eqnarray}
letting $z=-q$ again. Writing $(-1)^j
z^j-z^{j+1}=(1-z^{j+1})-(1-(-1)^j z^j)$, the summation on the
right-hand side of Equation~\eqref{eq:sgnmaj} becomes a
telescopic sum
\begin{multline}\sum_{j=1}^{n-2}\left(\prod_{i=1}^{j-1}(1+z^i)\prod_{i=j+1}^{n-1}(1-z^{i})-(1-(-1)^jz^j)\prod_{i=1}^{j-1}(1+z^i)\prod_{i=j+2}^{n-1}(1-z^{i})\right)\\
=\prod_{i=2}^{n-1}(1-z^i)+\sum_{\substack{j=2 \\
j\,\text{even}}}^{n-3}\left(2z^j\prod_{i=1}^{j-1}(1+z^i)\prod_{i=j+2}^{n-1}(1-z^{i})\right)-\prod_{i=1}^{n-2}(1+z^i),
\label{eq:secondsum}
\end{multline}
noting that $-(1-(-1)^jz^j)+(1+z^{j})=2z^j$ when $j$ is even and
$0$ otherwise. Now, writing
$$2z^j=\frac{(1+z^j)(1+z^{j+1})-(1-z^j)(1-z^{j+1})}{1+z},$$ the
summation in the middle of Equation~\eqref{eq:secondsum}
also becomes a telescopic sum
\begin{multline*}\frac{1}{1+z}\sum_{\substack{j=2 \\ j\,\text{even}}}^{n-3}\left(\prod_{i=1}^{j+1}(1+z^i)\prod_{i=j+2}^{n-1}(1-z^{i})-\prod_{i=1}^{j-1}(1+z^i)\prod_{i=j}^{n-1}(1-z^{i})\right)\\
=\frac{1}{1+z}\left(-(1+z)\prod_{i=2}^{n-1}(1-z^{i})+(1-z^{n-1})\prod_{i=1}^{n-2}(1+z^i)\right)=-\prod_{i=2}^{n-1}(1-z^{i})+(1-z^{n-1})\prod_{i=2}^{n-2}(1+z^i).
\end{multline*}
With these simplifications, Equation~\eqref{eq:sgnmaj} equals
\begin{multline*}\prod_{i=1}^{n-1}(1+z^i)+
\prod_{i=2}^{n-1}(1-z^i)-\prod_{i=2}^{n-1}(1-z^{i})+(1-z^{n-1})\prod_{i=2}^{n-2}(1+z^i)-\prod_{i=1}^{n-2}(1+z^i)\\
=\prod_{i=1}^{n-2}(1+z^i)\left(1+z^{n-1}+\frac{1-z^{n-1}}{1+z}-1\right)=[n]_{-z}\prod_{i=1}^{n-2}(1+z^i)=[n]_{q}\prod_{i=1}^{n-2}(1+(-q)^i).
\end{multline*}
\end{proof}

A different approach to prove Theorem~\ref{A:signed-maj} will be
described in Section~\ref{sec:another_approach}.


\section{The hyperoctahedral group: preliminaries and notation}\label{intro}

The hyperoctahedral group $B_n$
may be realized as a group of signed
permutations as follows. We denote by $B_{n}$ the group of all bijections $\pi$ of the
set $[\pm n]=\{-1,-2,\dots,-n,1,2,\dots,n\}$ onto itself such that
$$
\pi(-a) = -\pi(a)
$$
for every $1\le a\le n$, with composition as the group operation. This group is usually
known as the group of {\em signed permutations} on $\{1,2,\dots,n\}$, or as the
{\em hyperoctahedral group} of rank $n$. We identify $\S_{n}$ as a
subgroup of $B_{n}$, and $B_{n}$ as a subgroup of $\S_{2n}$ in
the natural ways.

If $\pi \in B_{n}$, we write $\pi = [a_{1}, \ldots ,a_{n}]$
to mean that $\pi(i)=a_{i}$ for $1\le i\le n$.
The Coxeter generating set of $B_n$ is $\s=\{\sigma_i:\ 0\le
i<n\}$, where $\sigma_0=[-1,2,3,4,\dots,n]$ and, for $1\le i<n$,
$\sigma_i$ is the adjacent transposition $(i,i+1)$.

\medskip

We recall some statistics on $B_n$. For $\pi\in B_n$, we say that $i$
is a descent in $\pi$ if $\pi(i)>\pi(i+1)$ with respect to
the order $-1<-2<\cdots<-n<1<2<\cdots<n$. We use the following
standard notation:
\begin{align*}
\D(\pi) &:=\{1\le i\le n-1:\ \pi(i)>\pi(i+1)\},\\
\des(\pi)&:=|\D(\pi)|,\\
\maj(\pi)&:=\sum\limits_{i\in \D(\pi)} i,\\
\Neg(\pi)&:=\{1\le i\le n:\ \pi(i)<0\},\\
\nega(\pi)&:=|\Neg(\pi)|.
\end{align*}
Two more statistics, defined in~\cite{AR-fm} and~\cite{ABR},
respectively, are the {\em flag-major index}
$$
\fmaj(\pi):= 2\cdot \maj(\pi)+\nega(\pi),
$$
and {\em flag-descent number}
$$
\fdes(\pi):=2\cdot \des(\pi)+\delta(\pi(1)<0),
$$
where $\delta(a):=1$ if the event $a$ occurs and zero otherwise.

\medskip

The statistics $\fmaj$ and $\fdes$ have been shown to play a
significant role in the study of $B_n$, which is analogous to the
role of the classical descent statistics on $\S_n$. Some examples in the literature
are \cite{ABR1, AGR, BRS, BeBr, CG, FH1, FH2, FH3}.

\section{Signed arc permutations}\label{sec:signed-arc}

\subsection{Definition and basic properties}

In this section we introduce our first generalization of arc permutations to type $B$.

\begin{definition}\label{def:signed_arc}
A permutation  $\pi=[\pi(1),\dots,\pi(n)]\in B_n$ is a {\em signed
arc permutation} if, for every $1< i<n$,
\begin{itemize}
\item the prefix $\{|\pi(1)|,\dots,|\pi(i)|\}$ forms an interval
in $\bbz_n$; and
\item the sign of $\pi(i)$ is positive if
$|\pi(i)|-1\in \{|\pi(1)|,\dots,|\pi(i-1)|\}$ and negative if
$|\pi(i)|+1\in \{|\pi(1)|,\dots,|\pi(i-1)|\}$ (with addition in $\bbz_n$).
\end{itemize}
Denote by $\As_n$ the set of signed arc permutations in $B_n$.
\end{definition}

Note that there is no restriction on the signs of $\pi(1)$ and $\pi(n)$.

\begin{example}We have that $[2,-1,3]\in\As_3$ and $[-3,-2,4,1]\in\As_4$, but $[-2,1,3]\notin\As_3$.
\end{example}

For $\pi\in B_n$, let $|\pi|=|\pi(1)||\pi(2)|\dots|\pi(n)|\in\S_n$.
Note that if $\pi\in\As_n$, then $|\pi|\in\A_n$.

\begin{remark}
{\rm The apparent ad-hoc determination of the signs in the second
part of Definition~\ref{def:signed_arc} surprisingly results in a
coherent combinatorial structure to be described below, which
further leads to interesting quasi-symmetric functions of type $B$
to be discussed in a forthcoming paper.
}
\end{remark}

\begin{claim}
For $n\ge1$, $|\As_n|=n2^n$.
\end{claim}

\begin{proof}
The equality is trivial for $n=1$, so we may assume that $n\ge2$.
From every $\sigma\in\A_n$, there are four
permutations $\pi\in\As_n$ such that $|\pi|=\sigma$, since all the signs
but those of the first and the last entry are determined.
It follows that $|\As_n|=4|\A_n|=n2^n$.
\end{proof}

\subsection{Characterization by pattern avoidance}

Let us recall the standard definition of pattern avoidance in the hyperoctahedral group. Given $\pi=[\pi(1),\dots,\pi(n)]\in B_n$ and
$\sigma=[\sigma(1),\dots,\sigma(k)]\in B_k$, we say that $\pi$ contains the pattern $\sigma$
if there exist indices $1\le i_1<\dots<i_k\le n$ such that
\begin{itemize}
\item $\pi(i_j)$ and $\sigma(j)$ have the same sign for all $1\le j\le k$, and
\item $|\pi(i_1)||\pi(i_2)|\dots|\pi(i_k)|$ is in the same relative order as
 $|\sigma(1)||\sigma(2)|\dots|\sigma(k)|$.
\end{itemize}
In this case, $\pi(i_1)\pi(i_2)\dots\pi(i_k)$ is called an {\em occurrence} of $\sigma$.
Otherwise, we say that $\pi$ {\em avoids} $\sigma$. For example, $[-3,2,5,-1,4]$ contains the pattern $[-2,-1,3]$, because the subsequence $-3,-1,4$ is an occurrence of this pattern, but it avoids the pattern $[2,1,3]$.

In analogy with Theorem~\ref{arc-pattern-thm} for arc permutations in $\S_n$, we can characterize signed arc permutations in terms of pattern avoidance.

\begin{theorem}\label{signed_arc_pattern}
A permutation $\pi\in B_n$ is a signed arc permutation if and only if it avoids the following $24$ patterns:
$$[\pm1,-2,\pm3],[\pm1,3,\pm2],[\pm2,-3,\pm1],[\pm2,1,\pm3],[\pm3,-1,\pm2],[\pm3,2,\pm1].$$
\end{theorem}

We say that a triple $(a,b,c)$ of different integers in $\{1,2,\dots,n\}$ is a {\em clockwise} triple if either $a<b<c$, $b<c<a$ or $c<a<b$. Otherwise, we say that it is a {\em counterclockwise} triple. The name comes from the direction determined by the triple $(a,b,c)$ in the circle where the entries $1,2,\dots, n$ have been written in clockwise order.

Note that the patterns listed in Theorem~\ref{signed_arc_pattern} are precisely  those permutations in $B_3$ of the form $[\pm a,-b,\pm c]$ where $(a,b,c)$ is a clockwise triple, and $[\pm a,b,\pm c]$ where $(a,b,c)$ is a counterclockwise triple.

\begin{proof}[Proof of Theorem~\ref{signed_arc_pattern}]
In this proof, addition and subtraction are in $\bbz_n$, and so are intervals.

Let $\pi\in B_n$ contain an occurrence $\pi(i_1)\pi(i_2)\pi(i_3)$ of one of the $24$ listed patterns. Suppose for contradiction that $\pi\in\As_n$.

If $\pi(i_2)>0$, then $(|\pi(i_1)|,|\pi(i_2)|,|\pi(i_3)|)$ is a
counterclockwise triple. Since $\pi\in\As_n$ and $\pi(i_2)$ is
positive, the interval $\{|\pi(1)|,\dots,|\pi(i_2-1)|\}$ contains
$|\pi(i_2)|-1$ and $|\pi(i_1)|$, but not $|\pi(i_2)|$. Thus, it
must also contain $|\pi(i_3)|$, which is a contradiction.

Similarly, if $\pi(i_2)<0$, then $(|\pi(i_1)|,|\pi(i_2)|,|\pi(i_3)|)$ is a clockwise triple, and the interval $\{|\pi(1)|,\dots,|\pi(i_2-1)|\}$ contains
$|\pi(i_2)|+1$  and $|\pi(i_1)|$, but not $|\pi(i_2)|$. Thus, it must also contain $|\pi(i_3)|$, again a contradiction.

\medskip

To prove the converse, suppose now that $\pi\in B_n$ is not a signed arc permutation.
Let $i$ be the smallest index where the conditions from Definition~\ref{def:signed_arc} fail.
This means that either $\{|\pi(1)|,\dots,|\pi(i)|\}$ is not an interval in $\bbz_n$, or $\pi(i)$ has the wrong sign. In the first case, neither of the values $|\pi(i)|\pm 1$  is in the interval $\{|\pi(1)|,\dots,|\pi(i-1)|\}$. In the second case, either $\pi(i)>0$ but $|\pi(i)|-1\notin \{|\pi(1)|,\dots,|\pi(i-1)|\}$, or  $\pi(i)<0$ but $|\pi(i)|+1\notin \{|\pi(1)|,\dots,|\pi(i-1)|\}$.

If $\pi(i)>0$ (respectively, $\pi(i)<0$), let $j>i$ be such that $|\pi(j)|=|\pi(i)|-1$ (respectively, $|\pi(j)|=|\pi(i)|+1$). Then $(|\pi(1)|,|\pi(i)|,|\pi(j)|)$ is a counterclockwise (respectively, clockwise) triple, so $\pi(1)\pi(i)\pi(j)$ is an occurrence of one of the $24$ listed patterns.
\end{proof}

\subsection{Descent set enumerators}

Next we describe the joint distribution of the descent set and the set of negative entries on signed arc permutations.

\begin{theorem}\label{signed_Des_Neg_inv}
For every $n\ge1$,
\beq\label{eq:signed_Des_Neg}
\sum_{\pi\in\As_n} \x^{\D(\pi)} \y^{\Neg(\pi)}=\prod_{i=1}^{n}(1+x_{i-1}y_{i})\left(1+\sum_{j=1}^{n-1} \frac{(x_{j}+x_{j-1}y_{j})(1+y_{j+1})}{(1+x_{j-1}y_{j})(1+x_{j}y_{j+1})}\right),
\eeq
and
\begin{multline}\label{eq:signed_Des_Neg_inv}
\sum_{\pi\in\As_n} t^{\inv(|\pi|)} \x^{\D(\pi)} \y^{\Neg(\pi)}=\prod_{i=1}^{n}(1+t^{i-1}x_{i-1}y_{i})\\
+\sum_{j=1}^{n-1}\left( (x_{j}+t^{j-1}x_{j-1}y_{j})(t^{j(n-j)}+t^{n-j-1}y_{j+1})\prod_{i=1}^{j-1}(1+t^{i-1}x_{i-1}y_{i})
\prod_{i=j+2}^{n}(1+t^{n-i}x_{i-1}y_{i})\right),
\end{multline}
with the convention that $x_0:=1$.
\end{theorem}

\begin{proof}
Since \eqref{eq:signed_Des_Neg} follows from \eqref{eq:signed_Des_Neg_inv} by setting $t=1$ and simplifying, it suffices to prove \eqref{eq:signed_Des_Neg_inv}.

If $\pi\in\As_n$ is such that $|\pi|$ is left-unimodal, then $|\pi(n)|\in\{1,n\}$. Let us
first consider signed arc permutations where $|\pi|$ is left-unimodal and $\pi(n)\in\{-1,n\}$.
The contribution of such permutations to the generating function is
$$\prod_{i=1}^{n}(1+t^{i-1}x_{i-1}y_{i}).$$
Indeed, such permutations are uniquely determined by a choice of sign of $\pi(i)$ for $1\le i\le n$.
If $\pi(i)$ is negative, it creates a descent with $\pi(i-1)$ (for $i>1$) and inversions in $|\pi|$ with all the preceding entries, contributing a factor $t^{i-1}x_{i-1}y_i$. If $\pi(i)$ is positive,
then no descent or inversions with preceding entries are created.

Let us now consider the remaining permutations $\pi\in\As_n$, and let $j+1$ be the first index where $\pi$ fails to be in the set considered above. In other words, if $|\pi|$ is not left-unimodal, $j$ is the largest such that $\{|\pi(1)|,\dots,|\pi(j)|\}$ is an interval in $\bbz$, and note that $1\le j\le n-2$ in this case.
On the other hand, if $|\pi|$ is left-unimodal but $\pi(n)\in\{1,-n\}$, then $j=n-1$. Consider two cases depending on the sign of $\pi(j+1)$.

\begin{itemize}
\item If $\pi(j+1)$ is positive, we must have $\pi(j+1)=1$. In this case, the first $j$ entries in $|\pi|$ are larger than the last $n-j$ entries, creating $j(n-j)$ inversions.
The contribution of permutations where $\pi(j)$ is positive as well (and so $\pi(j)=n$) is then
\beq\label{eq:piece} t^{j(n-j)} x_j\cdot\prod_{i=1}^{j-1}(1+t^{i-1}x_{i-1}y_{i})\prod_{i=j+2}^{n}(1+t^{n-i}x_{i-1}y_{i}).\eeq
To see this, first notice that the factor $x_j$ records the descent in position $j$. For $1\le i\le j-1$, each negative entry $\pi(i)$ creates a descent with $\pi(i-1)$ and inversions with all the preceding entries in $|\pi|$, contributing $t^{i-1}x_{i-1}y_{i}$.
For $j+2\le i\le n$, each negative entry $\pi(i)$ creates a descent with $\pi(i-1)$ and inversions with all the following entries in $|\pi|$, contributing $t^{n-i}x_{i-1}y_{i}$. In both cases, positive entries $\pi(i)$ just contribute a factor of $1$.

The contribution of permutations where $\pi(j)$ is negative is given by replacing $t^{j(n-j)} x_j$ with
$t^{j(n-j)} t^{j-1} x_{j-1}y_j$ in Equation~\eqref{eq:piece}, since now $\pi$ has a descent in position $j-1$, and $|\pi(j)|$ creates inversions with all the preceding entries in $|\pi|$.

\item If $\pi(j+1)$ is negative, we must have $\pi(j+1)=-n$. In this case, the contribution of permutations where $\pi(j)$ is negative (and so $\pi(j)=-1$) is
given by replacing $t^{j(n-j)} x_j$ with $t^{j-1} t^{n-j-1} x_{j-1}y_j y_{j+1}$ in Equation~\eqref{eq:piece}. Indeed, $\pi$ has a descent in position $j-1$, and there are inversions in $|\pi|$ between $|\pi(j)|$
and all the preceding entries, and between $|\pi(j+1)|$ and all the following entries.

Similarly, the contribution of permutations where $\pi(j)$ is positive is obtained by replacing $t^{j(n-j)} x_j$ with
$t^{n-j-1} x_{j}y_{j+1}$ in Equation~\eqref{eq:piece}, since $\pi$ has a descent in position $j-1$, and there are inversions in $|\pi|$ between $|\pi(j)|$
and all the preceding entries, and between $|\pi(j+1)|$ and all the following entries.
\end{itemize}

Adding all of the above contributions we obtain the stated formula.
\end{proof}


\subsection{The $(\fdes,\fmaj)$-enumerator}

\begin{corollary}\label{signed_arc_fmaj}
For every $n\ge2$,
$$\sum_{\pi\in\As_n} t^{\fdes(\pi)} q^{\fmaj(\pi)}=
(1+tq)\left(1+tq(1+q)+2t^2q^3[2n-3]_q+t^3q^{2n}(1+q)+t^4q^{2n+2}\right)\prod_{i=3}^{n-1} (1+t^2q^{2i-1}).$$
In particular,
$$\sum_{\pi\in\As_n}t^{\fdes(\pi)}=(1+t)(1+t^2)^{n-3}(1+2t+(4n-6)t^2+2t^3+t^4).$$
\end{corollary}

\begin{proof}
Substituting $y_1=tq$, $y_i=q$ for $2\le i\le n$, and
$x_i=t^2q^{2i}$ for $1\le i\le n-1$ in
Equation~\eqref{eq:signed_Des_Neg}, we get
$$\sum_{\pi\in\As_n} t^{\fdes(\pi)} q^{\fmaj(\pi)}=(1+tq)\prod_{i=2}^{n}(1+t^2q^{2i-1})\left(1+\frac{tq(1+q)}{1+t^2q^3}+\sum_{j=2}^{n-1} \frac{(1+q^2)t^2q^{2j-1}}{(1+t^2q^{2j-1})(1+t^2q^{2j+1})}\right).$$
Using that
$$\frac{t^2q^{2j-1}}{(1+t^2q^{2j-1})(1+t^2q^{2j+1})}=\frac{1}{1-q^2}\left(\frac{1}{1+t^2q^{2j+1}}-\frac{1}{1+t^2q^{2j-1}}\right),$$
the summation on the right-hand side becomes a telescopic sum that simplifies to
$$\frac{(1+q)t^2q^3[2n-4]_q}{(1+t^2q^3)(1+t^2q^{2n-1})}.$$
The first formula in the statement follows now from straightforward simplifications, and the second
formula is obtained by substituting $q=1$.
\end{proof}

\subsection{The signed $\fmaj$-enumerator}

Recall that $B_n$ has four one-dimensional characters: the trivial
character; the sign character $\sgn(\pi)$; $(-1)^{\nega(\pi)}$; and the
sign of $|\pi| \in \S_n$, denoted $\sgn(|\pi|)$.
Let us now compute the enumerators for signed arc permutations with respect to $\fmaj$ and each one of these characters.

\begin{corollary}\label{fmaj_signed_enumerate_s_arc}
For every $n\ge1$,
\begin{align}
\label{eq:signed_fmaj}
\sum_{\pi\in\As_n} q^{\fmaj(\pi)}&=[2n]_q \prod\limits_{i=1}^{n-1} (1+q^{2i-1}),\\
\label{eq:signed-signfmaj}
\sum_{\pi\in\As_n} \sgn(\pi) q^{\fmaj(\pi)}&=
\begin{cases}\displaystyle (1-q)[n]_{-q^2}\prod_{i=1}^{n-1}(1+(-1)^i q^{2i-1}) & \mbox{if $n$ is odd,} \\
\displaystyle [2n]_{q}\prod_{i=1}^{n-1}(1+(-1)^i q^{2i-1}) & \mbox{if $n$ is even,}\end{cases}\\
\nn
\sum_{\pi\in\As_n} (-1)^{\nega(\pi)} q^{\fmaj(\pi)}&=[2n]_{-q} \prod\limits_{i=1}^{n-1} (1-q^{2i-1}),\\
\nn
\sum_{\pi\in\As_n} \sgn(|\pi|) q^{\fmaj(\pi)}&=\begin{cases}\displaystyle (1+q)[n]_{-q^2}\prod_{i=1}^{n-1}(1+(-1)^{i-1} q^{2i-1}) & \mbox{if $n$ is odd,} \\ \displaystyle [2n]_{-q}\prod_{i=1}^{n-1}(1+(-1)^{i-1} q^{2i-1}) & \mbox{if $n$ is even.}\end{cases}
\end{align}
\end{corollary}

\begin{proof}
Equation~\eqref{eq:signed_fmaj} for $n\ge2$ is obtained from Corollary~\ref{signed_arc_fmaj} by substituting $t=1$, and it is trivial for $n=1$.

To prove~\eqref{eq:signed-signfmaj}, we use that the sign of $\pi\in B_n$ can be expressed as $\sgn(\pi)= (-1)^{\nega(\pi)} \sgn(|\pi|)=(-1)^{\inv(|\pi|)+\nega(\pi)}$.
Substituting $t=-1$, $y_i=-q$ and $x_i=q^{2i}$ for all $i$ in Theorem~\ref{signed_Des_Neg_inv}, we obtain
\begin{multline}\label{eq:signed-subs}\sum_{\pi\in\As_n} \sgn(\pi) q^{\fmaj(\pi)}=\sum_{\pi\in\As_n} (-1)^{\inv(|\pi|)+\nega(\pi)} q^{\fmaj(\pi)}=\prod_{i=1}^{n}(1+(-1)^iq^{2i-1})\\
+\sum_{j=1}^{n-1}\left( q^{2j-1}(q+(-1)^j)((-1)^{j(n-j)}+(-1)^{n-j}q)\prod_{i=1}^{j-1}(1+(-1)^iq^{2i-1})
\prod_{i=j+2}^{n}(1+(-1)^{n-i+1}q^{2i-1})\right).\end{multline}

When $n$ is odd, the right-hand side of Equation~\eqref{eq:signed-subs} simplifies to
$$\prod_{i=1}^{n}(1+(-1)^iq^{2i-1})\left(1+\sum_{j=1}^{n-1} \frac{(-1)^j q^{2j-1} (1-q^2)}{(1+(-1)^jq^{2j-1})(1+(-1)^{j+1}q^{2j+1})}\right).$$
The summation in the above formula can be written as a telescopic sum
$$\frac{1-q^2}{1+q^2}\,\sum_{j=1}^{n-1}\left(\frac{1}{1+(-1)^{j+1}q^{2j+1}}-\frac{1}{1+(-1)^jq^{2j-1}}\right)=\frac{q(1+q)((-1)^{n-1}q^{2n-2}-1)}{(1+q^2)(1+(-1)^n q^{2n-1})},$$
from where we obtain the expression in the statement.

When $n$ is even, using the shorthand $a_j=\prod_{i=1}^{j}(1+(-1)^iq^{2i-1})$ and $b_j=\prod_{i=j}^{n}(1+(-1)^{i-1}q^{2i-1})$, we can write the right-hand side of Equation~\eqref{eq:signed-subs} as
\beq\label{eq:sums1}a_n+\sum_{j=1}^{n-1} q^{2j-1}(1+(-1)^jq^2)a_{j-1}b_{j+2}+\sum_{j=1}^{n-1} q^{2j-1}(1+(-1)^jq)a_{j-1}b_{j+2}.\eeq
Using that $q^{2j-1}(1+(-1)^jq^2)=(1+(-1)^jq^{2j+1})-(1-q^{2j-1})$, the first summation in Equation~\eqref{eq:sums1} simplifies as a telescopic sum
$$\sum_{j=1}^{n-1} \left(a_{j-1}b_{j+1}-(1-q^{2j-1})a_{j-1}b_{j+2}\right)=b_2-a_{n-1}+\sum_{\substack{j=2 \\ j\,\text{even}}}^{n-2}2q^{2j-1}a_{j-1}b_{j+2}.$$
Combining this expression with the fact that the second summation in Equation~\eqref{eq:sums1} can be written as
$$\sum_{\substack{j=2 \\ j\,\text{even}}}^{n-2} 2q^{2j}a_{j-1}b_{j+2},$$
Equation~\eqref{eq:sums1} equals
\beq\label{eq:sums2}a_n+b_2-a_{n-1}+(1+q)\sum_{\substack{j=2 \\ j\,\text{even}}}^{n-2}2q^{2j-1}a_{j-1}b_{j+2}.\eeq
Now, using that $$2q^{2j-1}=\frac{(1-q^{2j+1})(1+q^{2j-1})-(1+q^{2j+1})(1-q^{2j-1})}{1-q^2},$$ Equation~\eqref{eq:sums2} simplifies to
\begin{multline*}a_n+b_2-a_{n-1}+\frac{1}{1-q}\sum_{\substack{j=2 \\ j\,\text{even}}}^{n-2}(a_{j+1}b_{j+2}-a_{j-1}b_j)=a_n+b_2-a_{n-1}+\frac{a_{n-1}b_n-a_1b_2}{1-q}\\
=a_{n-1}\left((1+q^{2n-1})-1+\frac{b_n}{1-q}\right)+b_2\left(1-\frac{a_1}{1-q}\right)=[2n]_q\, a_{n-1}=[2n]_{q}\prod_{i=1}^{n-1}(1+(-1)^i q^{2i-1}),
\end{multline*}
as claimed.

Finally, the generating functions $\sum_{\pi\in\As_n} (-1)^{\nega(\pi)} q^{\fmaj(\pi)}$ and $\sum_{\pi\in\As_n} \sgn(|\pi|) q^{\fmaj(\pi)}$ are easily obtained by replacing $q$ with $-q$ in Equation~\eqref{eq:signed_fmaj} and in Equation~\eqref{eq:signed-signfmaj}, respectively.
\end{proof}

\section{$B$-arc permutations}

\subsection{Definition and basic properties}

In this section we introduce a different generalization of arc permutations to type $B$.

Let $\O_n$ be a circle with $2n$ points labeled $-1,-2,\dots,-n,1,2,\dots,n$ in clockwise order, as shown in Figure~\ref{fig:On}. One can think of these points as
the elements of $\bbz_{2n}$, where for every $1\le j\le n$, the letter $-j$ is identified with $n+j\in\bbz_{2n}$.

\begin{figure}[htb]
  \begin{center}
    \begin{tikzpicture}[scale=0.4]
    \draw (0,0) circle (3);
    \draw[fill] (72:3) circle (.1); \draw(72:3) node[above right]{$1$};
    \draw[fill] (36:3) circle (.1); \draw(36:3) node[right]{$2$};
    \draw[fill] (0:3) circle (.1);
    \draw[fill] (-36:3) circle (.1);
    \draw[fill] (-72:3) circle (.1); \draw(-72:3) node[below right]{$n$};
    \draw[fill] (-108:3) circle (.1); \draw(-108:3) node[below left]{$-1$};
    \draw[fill] (-144:3) circle (.1); \draw(-144:3) node[left]{$-2$};
    \draw[fill] (180:3) circle (.1);
    \draw[fill] (144:3) circle (.1);
    \draw[fill] (108:3) circle (.1); \draw(108:3) node[above left]{$-n$};
    \end{tikzpicture}
  \end{center}
  \caption{The circle $\O_n$.}\label{fig:On}
\end{figure}

\begin{definition}\label{def:B-arc}
A permutation $\pi=[\pi(1),\dots,\pi(n)]\in B_n$ is a {\em $B$-arc permutation} if, for every $1\le j\le n$, the suffix $\{\pi(j),\pi(j+1),\dots,\pi(n)\}$
forms an interval in $\O_n$.
Denote by $\AB_n$ the set of $B$-arc permutations in $B_n$.
\end{definition}

\begin{example} We have that $[-2,3,-1]\in\AB_3$ and $[2,-1,4,3]\in\AB_4$, but $[-3,-1,2]\notin\AB_3$ and $[5,2,-1,4,3]\notin\AB_5$.
\end{example}

\begin{remark}{\rm While for permutations in $\S_n$ the suffix is an interval if and only if the prefix is an interval, this is not the case in $B_n$.
If we replaced {\em suffix} with {\em prefix} in Definition~\ref{def:B-arc}, the corresponding formulas in
Sections~\ref{approach-B} and~\ref{descents-B-arc} would be less elegant.}
\end{remark}

\begin{claim}\label{B-arc-enumeration}
For $n\ge1$, $|\AB_n|=n2^{n}$.
\end{claim}

\begin{proof}
Writing the entries of $\pi\in\AB_n$ from right to left, there are $2n$ choices for $\pi(n)$, and $2$ choices for each entry thereafter, since every suffix has to be be an interval in $\O_n$.
\end{proof}

\subsection{Characterization by pattern avoidance}

Paralleling our results for signed arc permutations, we can characterize $B$-arc permutations in terms of pattern avoidance.

\begin{theorem}\label{B_arc_pattern}
A permutation $\pi\in B_n$ is a $B$-arc permutation if and only if it avoids the following $24$ patterns:
\begin{align*} & [\pm2,1,3],[\pm2,3,1],[\pm3,1,-2],[\pm3,-2,1],[\pm1,2,-3],[\pm1,-3,2],\\
& [\pm2,-1,-3],[\pm2,-3,-1],[\pm3,-1,2],[\pm3,2,-1],[\pm1,-2,3],[\pm1,3,-2].
\end{align*}
\end{theorem}

Note that the patterns listed in the above theorem are precisely those of the form $[a,b,c]\in B_3$ where $b$ and $c$ are at distance at least $2$ in the circle $\O_3$.

\begin{proof}
If $\pi\in B_n$ contains an occurrence $\pi(i_1)\pi(i_2)\pi(i_3)$ of a pattern $[a,b,c]\in B_3$, where $b$ and $c$ are at distance at least $2$ in $\O_3$,
then the suffix $\{\pi(i_2),\pi(i_2+1),\dots,\pi(n)\}$ is not an interval in $\O_n$, since it contains the letters $\pi(i_2)$ and $\pi(i_3)$, but neither of the letters $\pm\pi(i_1)$.
Thus, $\pi\notin\AB_n$.

For the converse, suppose now that $\pi\in B_n$ is not a $B$-arc permutation. Take the largest $j$ such that $\{\pi(j),\pi(j+1),\dots,\pi(n)\}$ is not an interval in $\O_n$. Then $\pi(j)$ is at distance
at least 2 from the interval $\{\pi(j+1),\pi(j+2),\dots,\pi(n)\}$ in the circle $\O_n$.
It follows that there is some value $1\le k\le n$ such that $\pm k\notin \{\pi(j),\pi(j+1),\dots,\pi(n)\}$, but any interval containing $\pi(j)$ and $\pi(j+1)$ must also contain either $k$ or $-k$.
Let $i$ be such that $\pi(i)=\pm k$, and note that $1\le i< j$. We claim that the subsequence $\pi(i)\pi(j)\pi(j+1)$ is an occurrence of one of the patterns in the statement.
Noticing that $\AB_n$ is invariant under left multiplication
by $[n,n-1,\dots,1]$, we can assume without loss of generality
that $|\pi(j)|<|\pi(j+1)|$.
Additionally, by symmetry (reversing the
signs if necessary), we can assume that $\pi(j)>0$. Now these are
the possibilities:
\begin{itemize}
\item if $0<\pi(j+1)$ then $0<\pi(j)<k<\pi(j+1)$, so
$\pi(i)\pi(j)\pi(j+1)$ is an occurrence of $[\pm2,1,3]$; \item if
$-k<\pi(j+1)<0$ then $\pi(j)<k$, so $\pi(i)\pi(j)\pi(j+1)$ is an
occurrence of $[\pm3,1,-2]$; \item if $\pi(j+1)<-k$ then
$\pi(j)>k$, so $\pi(i)\pi(j)\pi(j+1)$ is an occurrence of
$[\pm1,2,-3]$.
\end{itemize}
\end{proof}

\subsection{Canonical expressions and signed enumeration}
\label{sec:another_approach}

In this subsection we characterize arc permutations and $B$-arc
permutations in terms of their canonical expressions. This
characterization is then applied to derive the unsigned and signed
flag-major index enumerators.

\subsubsection{
The type $A$ case}\label{approach-A}

For a positive integer $1\le m<n$ let $c_m:=\sigma_{m}
\sigma_{m-1} \cdots \sigma_1=(m+1,m,\dots,2,1)$, in cycle notation. 
Every permutation $\pi\in \S_n$ has a unique expression
$$
\pi = c_{n-1}^{k_{n-1}}c_{n-2}^{k_{n-2}}\cdots c_{1}^{k_{1}},
$$
with $0\le k_i\le i$ for all $1\le i<n-1$. Recall
from~\cite{AR-fm} that
\begin{equation}\label{eq-coxeter-maj}
\maj(\pi)=\sum\limits_{i=1}^{n-1} k_i.
\end{equation}
Indeed, in the above expression for $\pi$, each multiplication by
$c_m$ from the left rotates the values $1,2,\dots,m+1$ cyclically.
Changing the value $1$ to $m+1$ has the effect of moving a descent
one position to the right, while the other descents remain
unchanged.

\begin{proposition}\label{coxeter-arc}
A permutation $\pi\in \S_n$ is an arc permutation if and only if
$$
\pi = c_{n-1}^{k_{n-1}}c_{n-2}^{k_{n-2}}\cdots c_{1}^{k_{1}},
$$
with $0\le k_{n-1}\le n-1$ and $k_i\in\{0,i\}$ for all $1\le i\le
n-2$.
\end{proposition}

\begin{proof}
First, notice that a permutation in $\S_n$ is an arc permutation
if and only if it may obtained by rotation of the values of  a
left-unimodal permutation, namely, $\pi = c_{n-1}^k u$ for some
$u\in \LL_n$. Next,  a permutation $u\in \S_n$ is left-unimodal if
and only if its inverse has descent set $\{1,2,\dots,j\}$ for some
$1\le j\le n$. Equivalently, its inverse may be obtained from a permutation whose inverse
is in $\LL_{n-1}$ by inserting the letter $n$ at the beginning or at the end.
Hence, by induction on $n$, we have that $u\in \LL_n$ if and only if it has the form
$$
u = c_{n-2}^{k_{n-2}}c_{n-2}^{k_{n-2}}\cdots c_{1}^{k_{1}},
$$
with $k_i\in\{0,i\}$ for all $1\le i\le n-2$.
\end{proof}

\bigskip


The above characterization can be used to give a short algebraic proof of Theorem~\ref{A:signed-maj}.

\begin{proof}[Alternate proof of Theorem~\ref{A:signed-maj}]. Let
$\chi$ be a one-dimensional character of the symmetric group~$\S_n$.
Let $K_n:=\{\mathbf{k}=(k_1,\dots,k_{n-1}):\ 0\le
k_{n-1}\le n-1,\ k_i\in\{0,i\}\mbox{ for }1\le i\le n-2\}$.

By Proposition~\ref{coxeter-arc} and Equation~\eqref{eq-coxeter-maj},
\begin{multline*}
\sum\limits_{\pi\in \A_n} \chi(\pi) q^{\maj(\pi)}=
\sum\limits_{\mathbf{k}\in K_n} \chi(c_{n-1}^{k_{n-1}}\cdots
c_{1}^{k_{1}})q^{\maj(c_{n-1}^{k_{n-1}}\cdots c_{1}^{k_{1}})}
=\sum\limits_{\mathbf{k}\in K_n} \chi(c_{n-1}^{k_{n-1}}\cdots c_{1}^{k_{1}})
q^{\sum k_i}\\
=\sum\limits_{\mathbf{k}\in K_n} \prod\limits_{i=1}^{n-1}\chi(c_{i})^{k_{i}}\, q^{\sum k_i}
=\sum\limits_{\mathbf{k}\in K_n} \prod\limits_{i=1}^{n-1}(\chi(c_{i})q)^{k_i} =
\sum\limits_{k_{n-1}=0}^{n-1}(\chi(c_{n-1})q)^{k_{n-1}}
\prod\limits_{i=1}^{n-2}(1+\chi(c_{i})^i q^{i})\\
=
\begin{cases}\displaystyle [n]_q\prod_{i=1}^{n-1}(1+ q^i) & \mbox{if $\chi$ is the trivial character,} \\
\displaystyle  [n]_{(-1)^{n-1}q} \prod\limits_{i=1}^{n-2} (1+
(-q)^i) & \mbox{if $\chi$ is the sign character.}\end{cases}
\end{multline*}
\end{proof}

\subsubsection{
The type $B$ case}\label{approach-B}

In analogy with the formulas in Corollary~\ref{fmaj_signed_enumerate_s_arc} for signed arc permutations,
in this section we give formulas enumerating $B$-arc permutations with respect to $\fmaj$ and each one of the four one-dimensional characters in type $B$.

We will use the following characterization of $B$-arc permutations, analogous to the characterization of arc permutations given in Proposition~\ref{coxeter-arc}.

For a positive integer $0\le m<n$ let now
$$c_m:=\sigma_{m}\sigma_{m-1} \cdots \sigma_1\sigma_0=[-(m+1),1,2,\dots,m,m+2,\cdots,n].$$
Note that $c_m=(m+1,m,\dots,1,-(m+1),-m,\dots,-1)$ in cycle notation, and it has order $2m+2$.
Every $\pi\in B_n$ has a unique expression
$$
\pi = c_{n-1}^{k_{n-1}}c_{n-2}^{k_{n-2}}\cdots c_{1}^{k_{1}}
c_0^{k_0},
$$
with $0\le k_i\le 2i+1$ for all $0\le i<n$.
Recall from~\cite{AR-fm} that
\begin{equation}\label{eq-coxeter-fmaj_B}
\fmaj(\pi)=\sum\limits_{i=0}^{n-1} k_i.
\end{equation}

\begin{proposition}\label{coxeter-arc-B}
A permutation $\pi\in B_n$ is a $B$-arc permutation if and only if
$$
\pi = c_{n-1}^{k_{n-1}}c_{n-2}^{k_{n-2}}\cdots
c_{1}^{k_{1}}c_0^{k_0},
$$
with $0\le k_{n-1}\le 2n-1$ and $k_i\in\{0,2i+1\}$ for all $0\le
i\le n-2$.
\end{proposition}

\begin{proof}
For every $0\le i<n$, if $\pi\in \AB_n$ is such that $\pi(j)=j$ for all $j>i$,
the permutation $c_i^{2i+1}\pi=c_i^{-1}\pi$ is also a
$B$-arc permutation. It follows by induction that $c_{n-2}^{k_{n-2}}\cdots
c_0^{k_0}\in\AB_n$ for all choices of $k_i\in\{0,2i+1\}$. Next,
notice that $\AB_n$ is invariant under left multiplication by $c_{n-1}$,
since this operation is a counterclockwise rotation of the letters in $\O_n$.
One concludes that
$$
\{c_{n-1}^{k_{n-1}}c_{n-2}^{k_{n-2}}\cdots c_0^{k_0}:\ 0\le k_n\le
2n-1\ \text{ and }\  k_i\in\{0,2i+1\} \ \text{ for all } 0\le i<n
\ \}\subseteq \AB_n.
$$
Finally, we prove that these two sets are equal by showing that they have the same cardinality.
The set on the left-hand side has size $n2^n$, because each choice of the $k_i$
yields a different element of $B_n$. By Claim~\ref{B-arc-enumeration},
this coincides with the cardinality of $\AB_n$.
\end{proof}

Product formulas for unsigned, signed and other one-dimensional
character enumerators for the flag-major index follow.

\begin{theorem}\label{B-enumerators}
For every $n\ge 1$,
\begin{align}
\sum\limits_{\pi\in \AB_n} q^{\fmaj(\pi)}&=[2n]_q
\prod\limits_{i=1}^{n-1} (1+q^{2i-1}),\label{eq:fmaj}\\
\nn
\sum\limits_{\pi\in \AB_n} \sgn(\pi) q^{\fmaj(\pi)}
&=[2n]_{(-1)^nq} \prod\limits_{i=1}^{n-1}
(1+(-1)^{i}q^{2i-1}),\\
\nn
\sum\limits_{\pi\in \AB_n} (-1)^{\nega(\pi)} q^{\fmaj(\pi)}
&=[2n]_{-q} \prod\limits_{i=1}^{n-1} (1-q^{2i-1}),\\
\nn
\sum\limits_{\pi\in \AB_n} \sgn(|\pi|) q^{\fmaj(\pi)}
&=[2n]_{(-1)^{n-1}q} \prod\limits_{i=1}^{n-1} (1+(-1)^{i-1}q^{2i-1}).
\end{align}
\end{theorem}

\begin{proof}
Let $\chi$ be a one-dimensional character of $B_n$. Let $K'_n:=\{\mathbf{k}=(k_0,k_1,\dots,k_{n-1}):\ 0\le
k_{n-1}\le 2n-1,\ k_i\in\{0,2i+1\}\mbox{ for }0\le i\le n-2\}$.

By Proposition~\ref{coxeter-arc-B} and Equation \eqref{eq-coxeter-fmaj_B},
\begin{multline*}
\sum\limits_{\pi\in \AB_n} \chi(\pi) q^{\fmaj(\pi)}=
\sum\limits_{\mathbf{k}\in K'_n} \chi(c_{n-1}^{k_{n-1}}\cdots
c_{1}^{k_{1}})q^{\fmaj(c_{n-1}^{k_{n-1}}\cdots c_{0}^{k_{0}})}
=\sum\limits_{\mathbf{k}\in K'_n} \chi(c_{n-1}^{k_{n-1}}\cdots c_{0}^{k_{0}})q^{\sum k_i}\\
 =\sum\limits_{\mathbf{k}\in K'_n} \prod\limits_{i=0}^{n-1}\chi(c_{i})^{k_{i}}\,q^{\sum k_i}
=\sum\limits_{\mathbf{k}\in K'_n} \prod\limits_{i=0}^{n-1}(\chi(c_{i})q)^{k_i} =
\sum\limits_{k_{n-1}=0}^{2n-1}(\chi(c_{n-1})q)^{k_{n-1}}
\prod\limits_{i=0}^{n-2}(1+\chi(c_{i})q^{2i+1}).
\end{multline*}
\end{proof}

\begin{remark}{\rm
A characterization similar to Proposition~\ref{coxeter-arc-B} holds for signed arc permutations.
Letting $t_n$ be the reflection $[1,2,\dots,n-1,-n]$, one can show
that a permutation $\pi\in B_n$ is a signed arc permutation if and
only if
$$
\pi = (c_{n-1}c_0)^{k_n} t_n^{k_{n-1}}c_{n-2}^{k_{n-2}}\cdots
c_{1}^{k_{1}}c_0^{k_0},
$$
with $0\le k_{n}\le n-1$ and $k_i\in\{0,-1\}$ for all $0\le i\le
n-1$.
Unlike in the case of $B$-arc permutations, we did not find this characterization helpful in
computing enumerators.}
\end{remark}

\subsection{The $(\fdes,\fmaj)$-enumerator}\label{descents-B-arc}

Next we apply a coset analysis to calculate the
bivariate $(\fdes,\fmaj)$-enumerator on $B$-arc permutations.

\begin{theorem}\label{B-prop:fmaj-fdes}  For every $n\ge2$,
\begin{align}
\sum_{\pi\in\AB_n}t^{\fdes(\pi)}q^{\fmaj(\pi)}&=\frac{(1+tq)(1+tq^n)}{1-q}\left((1-tq^n)\prod_{i=1}^{n-2}(1+t^2q^{2i+1})-(1-t)q\prod_{i=1}^{n-2}(1+t^2q^{2i+2})\right),\label{eq:fmaj-fdes}\\
\sum_{\pi\in\AB_n}t^{\fdes(\pi)}&=(1+t)^3 (1+t^2)^{n-3} (1+(n-2)t+t^2).\label{eq:fdes}
\end{align}
\end{theorem}

\begin{proof}
Fix $n$, and let $c:=c_{n-1}=[-n,1,2,\dots,n-1]=\sigma_{n-1}\sigma_{n-2}\cdots \sigma_0$, which is a
Coxeter element in $B_n$, and it has order $2n$. Recall that $\AB_n$ is
closed under left multiplication by $c$, which corresponds to shifting the values of $\pi$ one
position counterclockwise in $\O_n$. A collection of
representatives of the distinct left cosets of the cyclic subgroup
generated by $c$ is given by $\{\pi\in \AB_n: \pi(n)=n\}$.
Denoting this set by $\wt{\AB_n}$, we can write $\AB_n$ as a
disjoint union
$$
\AB_n=\bigcup_{j=0}^{2n-1} \{c^j\pi: \pi \in \wt{\AB_n}\}.
$$

Before proving Equation~\eqref{eq:fmaj-fdes}, we start with the case $t=1$ to illustrate our technique. This is Equation~\eqref{eq:fmaj}, which we proved above using a different method.
We first show that
$$
\sum_{\pi\in\wt{\AB_n}} q^{\fmaj(\pi)}=\prod\limits_{i=1}^{n-1}
(1+q^{2i-1}).
$$
Indeed, for every $1\le i< n$, given a suffix of $n-i$ letters,
which is an interval containing $n$, there are two choices for the preceding letter $\pi(i)$:
 positive and maximal among the remaining letters, or negative and minimal. In the first case, $\pi(i-1)$ must be smaller than $\pi(i)$ and the
contribution to the flag-major index is zero. In the second case,
since $\pi(i)$ is negative and minimal among the remaining
letters, $i-1$ must be a descent, and the contribution to the
flag-major index is $2(i-1)+1$.

It is easy to verify that for every $\pi\in \AB_n$ and
$0\le j<2n$, \beq\label{eq:fmajcpi} \fmaj(c^j \pi)= \fmaj(\pi)+j.
\eeq One concludes that
$$
\sum_{\pi\in\AB_n}q^{\fmaj(\pi)}=[2n]_q \sum_{\pi\in\wt{\AB_n}}
q^{\fmaj(\pi)},
$$
which implies equation~\eqref{eq:fmaj}.

Refining the above argument, we can enumerate permutations
$\pi\in\wt{\AB_n}$ with $\pi(1)>0$ according to the descent
set, the value of $\pi(1)$, and $\nega(\pi)$ as follows:
\begin{align*}
\sum_{\{\pi\in \wt{\AB_n}:\ \pi(1)>0\}} \x^{\D(\pi)}
y^{\pi(1)} z^{\nega(\pi)}&= y^{n-1}
\left(x_{n-2}z+\frac{1}{y}\right)\left(x_{n-3}z+\frac{1}{y}\right)\dots\left(x_{1}z+\frac{1}{y}\right)\\
&=y\prod_{i=1}^{n-2}(1+x_iyz).
\end{align*}
To see this, let $2\le i< n$, and suppose that the entries
$\pi(i+1),\pi(i+2),\dots,\pi(n)$ have been chosen, forming an
interval in $\O_n$ containing $n$. Suppose that this interval is
bounded by $-k<0$ and $m>0$. There are two choices for the entry
$\pi(i)$, namely $-k-1$ and $m-1$. If $\pi(i)=-k-1$, then
$\pi(i-1)\pi(i)$ will be a descent, regardless of how $\pi(i-1)$
is chosen, and additionally $\pi(i)$ contributes to $\nega(\pi)$.
On the other hand, if $\pi(i)=m-1$, then $\pi(i-1)\pi(i)$ will not
be a descent. Finally, there is only one choice for $\pi(1)$ once
$\pi(2),\pi(3),\dots,\pi(n)$ have been chosen, since $\pi(1)>0$,
and its value will be $n-1$ minus the number of indices $2\le i<
n$ for which the positive choice for $\pi(i)$ has been made.

Similarly, for permutations $\pi\in\wt{\AB_n}$ with
$\pi(1)<0$, we get
$$
\sum_{\{\pi\in \wt{\AB_n}:\ \pi(1)<0\}} \x^{\D(\pi)}
y^{|\pi(1)|} z^{\nega(\pi)}=yz\prod_{i=1}^{n-2}(1+x_iyz),
$$
and so
$$
\sum_{\pi\in \wt{\AB_n}} \x^{\D(\pi)} y^{|\pi(1)|}
z^{\nega(\pi)}
u^{\delta(\pi(1)<0)}=y(1+uz)\prod_{i=1}^{n-2}(1+x_iyz).
$$
Making the substitutions $x_i=t^2q^{2i}$, $z=q$ and $u=t$, we
obtain
$$
P(t,q,y):=\sum_{\pi\in \wt{\AB_n}} t^{\fdes(\pi)}
q^{\fmaj(\pi)} y^{|\pi(1)|}
=y(1+tq)\prod_{i=1}^{n-2}(1+yt^2q^{2i+1}).
$$

Given $\pi\in \wt{\AB_n}$ with $\pi(1)=a>0$, let us analyze
the values of $\fdes$ on the coset $\{c^j\pi: 0\le j<2n\}$. To see
how $\fdes$ changes when multiplying by $c$, note that
$\des(c\sigma)=\des(\sigma)$ unless $\sigma(1)=-1$, in which case
$\des(c\sigma)=\des(\sigma)+1$, or $\sigma(n)=-1$, in which case
$\des(c\sigma)=\des(\sigma)-1$. Thus,
$$\des(c^j\pi)=\begin{cases} \des(\pi) & \mbox{if }0\le j<n+a, \\
\des(\pi)+1 & \mbox{if }n+a\le j<2n. \end{cases}$$ Since
$c^j\pi(1)<0$ precisely for $a\le j<n+a$, it follows that
\begin{equation}\label{eq:fdescpi}
\fdes(c^j\pi)=\begin{cases} \fdes(\pi) & \mbox{if }0\le j<a, \\
\fdes(\pi)+1 & \mbox{if }a\le j<n+a, \\
\fdes(\pi)+2 & \mbox{if }n+a\le j<2n. \end{cases}\end{equation}

Similarly, given $\pi\in \wt{\AB_n}$ with $\pi(1)=a<0$, we
have
$$\des(c^j\pi)=\begin{cases} \des(\pi) & \mbox{if }0\le j<a, \\
\des(\pi)+1 & \mbox{if }a\le j<n+a, \\
\des(\pi) & \mbox{if }n+a\le j<2n, \end{cases}$$ and since
$c^j\pi(1)<0$ precisely when $0\le j<a$ or $n+a\le j<2n$, the same
formula~\eqref{eq:fdescpi} for $\fdes(c^j\pi)$ holds.

Using equations~\eqref{eq:fdescpi} and~\eqref{eq:fmajcpi}, we see
that if the contribution of $\pi\in \wt{\AB_n}$ to the
generating function $P(t,q,y)$ is $t^{\fdes(\pi)} q^{\fmaj(\pi)}
y^{|\pi(1)|}=t^d q^m y^a$, then the contribution of the coset
$\{c^j\pi: 0\le j<2n\}$ to the generating function
$\sum_{\pi\in\AB_n} t^{\fdes(\pi)}q^{\fmaj(\pi)}$ is
\begin{multline*}t^d q^m(1+ q+\dots+q^{a-1}+t q^{a}+t q^{a+1}+\dots+t q^{n+a-1}+ t^2 q^{n+a}+t^2 q^{n+a+1}+\dots+t^2 q^{2n-1})\\
=t^dq^m([a]_q+tq^a[n]_q+t^2q^{n+a}[n-a]_q)=t^d q^m
\frac{1-t^2q^{2n}-(1-t)(1+tq^n)q^a}{1-q}.
\end{multline*}
It follows that
\begin{align*}\sum_{\pi\in\AB_n} t^{\fdes(\pi)}q^{\fmaj(\pi)}&=
\frac{(1-t^2q^{2n})P(t,q,1)-(1-t)(1+tq^n)P(t,q,q)}{1-q}\\
&=\frac{(1+tq)(1+tq^n)}{1-q}\left((1-tq^n)\prod_{i=1}^{n-2}(1+t^2q^{2i+1})-(1-t)q\prod_{i=1}^{n-2}(1+t^2q^{2i+2})\right),
\end{align*}
proving~\eqref{eq:fmaj-fdes}.

When $q=1$, it is easy to realize that if the contribution of a
permutation $\pi\in \wt{\AB_n}$ to $P(t,1,y)$ is $t^d y^a$,
then the contribution of the coset $\{c^j\pi: 0\le j<2n\}$ to
$\sum_{\pi\in\AB_n} t^{\fdes(\pi)}$ is
$$t^d(a+t n+t^2 (n-a))=t^d(a(1-t^2)+nt(1+t)).$$
It follows that
\begin{align*}\sum_{\pi\in\AB_n} t^{\fdes(\pi)}&=
(1-t^2)\frac{\partial}{\partial y}P(t,1,y)\bigg|_{y=1}+nt(1+t)P(t,1,1)\\
&=(1-t^2)(1+t)(1+t^2)^{n-3}(1+(n-1)t^2)+nt(1+t)^2(1+t^2)^{n-2}\\
&=(1+t)^2(1+t^2)^{n-3}\left((1-t)(1+(n-1)t^2)+nt(1+t^2)\right)\\
&=(1+t)^3(1+t^2)^{n-3}(1+(n-2)t+t^2),
\end{align*}
proving~\eqref{eq:fdes}.

\end{proof}

\subsection{The descent set enumerator}


In this subsection we apply a descent-set preserving map
to reduce the calculation of the descent set enumerator on $B$-arc
permutations to the type $A$ case.

\begin{theorem}\label{Thm_AB_Des}
For every $n\ge2$, \begin{equation}\label{eq_Des}
\sum_{\pi\in\AB_n} \x^{\D(\pi)}=
\prod\limits_{i=1}^{n-1}(1+x_i)\left(2+n+
2\sum_{i=1}^{n-2}\frac{x_i+x_{i+1}}{(1+x_i)(1+x_{i+1})}\right).
\end{equation}
\end{theorem}

\begin{proof}
We show that there exists an $n$-to-$1$ descent-set
preserving map from the subset of permutations in $\AB_n$ which
contain the letter $1$ to $\LL_n$, and
a $2$-to-$1$ descent-set preserving map from permutations in $\AB_n$
which contain the letter $-1$ to $\A_n$.

For $1\le k\le n$, denote by $B_{n,k}$ the set of permutations in $\AB_n$ whose support
is $-k+1,-k+2\dots,-n,1,\dots,k$. Note that $\bigcup_{1\le k\le n}B_{n,k}$ is the
subset of permutations in $\AB_n$ which contain the entry $1$.
Permutations $\pi\in B_{n,k}$ are determined by choosing, for $1\le i\le n-1$, whether $\pi(i)$
is the largest or the smallest of the remaining entries. Clearly, $\pi(i)$ creates a descent with $\pi(i+1)$ only in the first case. It follows that
\beq\label{eq:Bnk}
\sum_{\pi\in B_{n,k}}
\x^{\D(\pi)}=\prod\limits_{i=1}^{n-1}(1+x_i),
\eeq
which, as shown in the proof of Theorem~\ref{thm:A-invDes}, coincides with the descent set enumerator on $\LL_n$. In fact, this construction gives a natural descent-set preserving bijection from $B_{n,k}$ to $\LL_n$,
and thus an $n$-to-$1$ descent-set preserving map from permutations in $\AB_n$ which contain $1$ to $\LL_n$.

Next we describe a $2$-to-$1$ descent-set preserving map from permutations
in $\AB_n$ which contain $-1$ to $\A_n$. The image of $\pi$ is simply defined to be $|\pi|$, that is,
the permutation obtained by forgetting the signs. It is easy to check that $|\pi|\in\A_n$ and that this map preserves the descent set.

To see that it is a $2$-to-$1$ map, we show that each permutation $[a_1, a_2,\dots, a_n]\in\A_n$ has
exactly two preimages. If $a_1\ne 1$, the preimages are $[a_1,a_2',a_3',\dots,a_n']$ and $[-a_1,a_2',a_3',\dots,a_n'$], where
$$a_i'=\begin{cases} a_i & \mbox{if } a_i>a_1,\\ -a_i & \mbox{otherwise.}\end{cases}$$
If $a_1=1$, the preimages are $[-1,a_2,\dots,a_n]$ and $[-1,-a_2,\dots,-a_n]$.

Combining Equation \eqref{eq:A-Des}, which gives the distribution of the descent set on $\A_n$, with
Equation~\eqref{eq:Bnk}, we conclude that
$$
\sum_{\pi\in\AB_n} \x^{\D(\pi)}= n\sum_{\pi\in\LL_n}
\x^{\D(\pi)}+2\sum_{\pi\in\A_n} \x^{\D(\pi)}= n
\prod\limits_{i=1}^{n-1}(1+x_i)+2\prod\limits_{i=1}^{n-1}(1+x_i)
\left(1+\sum_{i=1}^{n-2}\frac{x_i+x_{i+1}}{(1+x_i)(1+x_{i+1})}\right),
$$
which equals the right-hand side of \eqref{eq_Des}.
\end{proof}

\section{Final remarks and open problems}\label{Final}

Comparing Theorem~\ref{signed_Des_Neg_inv} with
Theorem~\ref{Thm_AB_Des}, we see that the descent set has
different distributions on $\As_n$ and $\AB_n$.
However, combining Theorem~\ref{B-enumerators} with
Corollary~\ref{fmaj_signed_enumerate_s_arc}, we obtain the following equidistribution phenomena.
It would be natural to look for bijective proofs.

\begin{corollary}\label{equi-B_vs_signed}
\begin{enumerate}
\item For every $n\ge 1$,
$$
\sum\limits_{\pi\in \AB_n} q^{\fmaj(\pi)}=\sum\limits_{\pi\in
\As_n} q^{\fmaj(\pi)}.
$$
\item For every even $n\ge2$,
$$
\sum\limits_{\pi\in \AB_n} \sgn(\pi)
q^{\fmaj(\pi)}=\sum\limits_{\pi\in \As_n} \sgn(\pi)
q^{\fmaj(\pi)}.
$$
\end{enumerate}
\end{corollary}

\bigskip

Signed arc permutations and $B$-arc permutations have further properties
analogous to those of unsigned arc permutations.
In particular, both sets carry affine Weyl group actions, interesting
underlying graph structures, and descent-set preserving maps to
standard Young tableaux.
Whereas the definition of $B$-arc permutations is more natural and
gives rise to a nicer underlying graph structure, signed arc
permutations have a finer joint distribution of the descent set
and the set of negative entries, which leads to interesting
quasi-symmetric functions of type $B$ to be discussed in a
forthcoming paper.

\bigskip

We conclude by mentioning a natural direction in which our work could be extended.
The flag-major index and flag-descent number have been generalized
to classical complex reflection groups in~\cite{C1, BC, BB, SAR}.

\begin{problem}
Generalize the concept of arc permutations to the 
complex reflection group $G(r,p,n)$.
\end{problem}

Finding elegant descent enumerators on these generalized arc permutations
may serve as an indicator of a ``correct'' generalization. It should be noted that
natural analogues of Equation \eqref{eq-coxeter-maj} and
Proposition~\ref{coxeter-arc} hold on wreath products
$G(r,1,n)=\bbz_r\wr \S_n$. Thus, enumerators on
$B$-arc permutations 
could be generalized to all one-dimensional character enumerators for the
flag-major index on these sets.

A more challenging task is to find a unified abstract
generalization of arc permutations to all Coxeter groups,
including 
affine as well as exceptional types.

\end{document}